\documentclass[a4paper,12pt]{amsart}
\usepackage{amsthm}
\usepackage{amsmath}
\usepackage{amsfonts, mathrsfs, mathabx}
\usepackage{amssymb, amscd, wasysym}
\usepackage{enumerate, url, hyperref}
\usepackage[all]{xy}
\title{Cohomology jump loci in the moduli spaces of vector bundles}
\begin{document}
\newtheorem{thm}{Theorem}[section]
\newtheorem{corollary}[thm]{Corollary}
\newtheorem{lemma}[thm]{Lemma}
\newtheorem{prop}[thm]{Proposition}
\newtheorem{convention}[thm]{Convention}
\newtheorem{remark}[thm]{Remark}
\newtheorem{notation}[thm]{Notation}
\author{Botong Wang}
\address{Department of Mathematics, University of Notre Dame, 255 Hurley Hall, IN 46556, USA}
\email{bwang3@nd.edu}
\date{}
\maketitle
\begin{abstract}
Two decades ago, as part of their work of generic vanishing theorems, Green-Lazarsfeld showed that over a compact K\"ahler manifold $X$, the cohomology jump loci in the $Pic^\tau(X)$ are all translates of subtori. In this paper, we generalize this result to higher dimensional vector bundles. It is showed by Nadel that locally the moduli space of vector bundles with vanishing chern classes is canonically isomorphic to a quadratic cone in the Zariski tangent space of a point. We prove that under the isomorphism, the cohomology jump loci are defined by linear equations.
\end{abstract}

\section{Introduction}
Let $X$ be a compact K\"ahler manifold. Green and Lazarsfeld studied the analytic subvarieties $\Sigma^i_k(X)\stackrel{\textrm{def}}{=}\{L\in Pic^\tau(X)\,| \, \dim H^i(X, L)\geq k\}$. They have showed the following two theorems.

\begin{thm}[Green-Lazarsfeld]
The following inequality holds for every $0\leq i\leq \dim(X)$
$$\textup{codim}(\Sigma^i_1(X), Pic^\tau(X))\geq \dim a(X)-i$$
where $a(X)$ is the image of $X$ under the Albanese map. 
\end{thm}

\begin{thm}[Green-Lazarsfeld]
All $\Sigma^i_k(X)$ are translates of subtori in $Pic^\tau(X)$.
\end{thm}

The second theorem was generalized and improved in different directions. In \cite{a}, \cite{b} and \cite{li0}, it was partially generalized to quasiprojective varieties. In \cite{s2}, it was proved that when $X$ is a smooth projective variety, $\Sigma^i_k$ are unions of torsion translates of subtori. 

In this paper, we generalize the second theorem to vector bundles of higher rank. Let $\mathcal{M}$ be the moduli space of stable vector bundles of rank $r$ and with vanishing chern classes. Since $\mathcal{M}$ is not reduced in general, the definition of the cohomology jump loci as analytic subspaces is not automatic any more. In the second section, we provide a rigorous definition of the cohomology jump loci $\Sigma^i_k(X)$ in $\mathcal{M}$ and its characterization. 

Let $E$ be a stable vector bundle corresponding to a point $P$ in $\mathcal{M}$. Nadel showed that near $P$, $\mathcal{M}$ is canonically isomorphic to a quadratic cone $C$ in $H^1(X, End(E))$. More precisely, there is a open neighborhood $U$ of $P$ in $\mathcal{M}$, which is isomorphic to an open neighborhood $U'$ of the origin in the cone $C$. And this isomorphism is induced by a canonical exponential map. In the third section, we prove the following main theorem,

\begin{thm}\label{linear}
There exists a linear subspace $H$ of $H^1(X, End(E))$, such that the following holds. When $U$ is sufficiently small, the isomorphism between $U$ and $U'$ induces an isomorphism between $\Sigma^i_k(X)\cap U$ and $H\cap U'$. 
\end{thm}

\begin{remark}
This vector space $H$ will be defined explicitly in the last section.
\end{remark}
The theorem implies the second theorem of Green-Lazarsfeld. In fact, in the rank one case, there are no higher obstructions for deformation of a line bundle. Hence, the cone is the whole $H^1(X, End(E))=H^1(X, \mathcal{O}_X)$, and locally $\mathcal{M}$ is isomorphic to $H^1(X, \mathcal{O}_X)$ via the inverse of the exponential map (up to a scalar, see Proposition 1.1 of \cite{gl2}). Now, take an irreducible component $Y$ of $\Sigma_k^i(X)$ in $Pic^\tau(X)$. Without loss of generality, we assume $Y$ is not contained in $\Sigma_{k+1}^i(X)$. Then a general point $P$ on $Y$ is not contained in $\Sigma_{k+1}^i(X)$. According to Theorem \ref{linear}, near $P$, $Y$ is the image of a linear subspace of $H^1(X, \mathcal{O}_X)$ under the exponential map. Therefore, $Y$ is a translate of subtorus. Hence, $\Sigma_k^i(X)$ is a union of translates of subtori. 

The main theorem also implies the next corollary about the intersections of those subtori, which is implicit in \cite{gl1} and \cite{gl2}.
\begin{corollary}
Assume $r=1$, that is, $\mathcal{M}=Pic^\tau(X)$. If $P$ is a singular point of $\Sigma^i_k(X)$. Then $P\in \Sigma^i_{k+1}(X)$.
\end{corollary}

For arbitrary $r$, it also gives a condition on the singularity of $\Sigma^i_k(X)$.
\begin{corollary}
For any point $P\in \Sigma^i_k(X)\setminus \Sigma^i_{k+1}(X)$, the analytic germ $(\Sigma^i_k(X))_P$ is quadratic.
\end{corollary}
This corollary was also obtained by Martinengo \cite{ma} using the differential graded Lie algebra method developed in \cite{gm}. 

We approach the problem by studying maps $s: Spec(A)\to \mathcal{M}$, where $A$ is an Artinian local ring over $\mathbb{C}$. From such a map $s$, we obtain a vector bundle $E_A$ over $X$ by pulling back the Kuranishi family on $X\times \mathcal{M}$. We reduce the problem of studying the cohomology jump loci $\Sigma^i_k(X)$ to the problem of studying the cohomology group $H^i(X, E_A)$ as an $A$-module. In \cite{li}, Libgober studied the tangent space of $\Sigma^i_k(X)$ by studying  $H^i(X, E_A)$ for $A=\mathbb{C}[\varepsilon]/(\varepsilon^2)$.

Under the philosophy of non-abelian Hodge theory (e.g. \cite{s1}), the moduli space of stable vector bundles with vanishing chern classes corresponds to the moduli space of irreducible unitary representations. The cohomology jump loci for higher rank unitary (and general algebraic group) representations was studied in \cite{dps} and \cite{dp}. They studied the cohomology jump loci near the trivial representation, and they obtained deformation theory of $\Sigma^i_k(X)$ even along $\Sigma^i_{k'}(X)$ for $k'>k$ by studying the resonance varieties. Compared to their work, our result concerns every point in the moduli space, not only near the trivial bundle, and it handles the non-reduced part of the moduli space. 

Throughout the paper, $X$ is a connected compact K\"ahler manifold and $E$ is a holomorphic vector bundle on $X$ of rank $r$. $End(E)$ is the holomorphic vector bundle $\mathscr Hom(E, E)$. By an Artinian local ring, we always mean an Artinian local ring which is of finite type over $\mathbb{C}$. 

\textbf{Aknowlegement.} I thank Prof. Arapura for introducing the results of \cite{gl1}, \cite{gl2} to me, and answering many of my questions. I also appreciate Prof. Budur for reading a previous version of this paper and giving many helpful suggestions.  And I thank Prof. Simpson for drawing my attention to some related works.

\section{Cohomology jump loci and its characterization}
First of all, we consider a rather simple situation. Let $M$ be an analytic space, and denote its structure sheaf by $\mathcal{O}_M$. Let
\begin{equation}\label{free}
\cdots\longrightarrow F^{i-1}\stackrel{d^{i-1}}{\longrightarrow}F^i\stackrel{d^i}{\longrightarrow}F^{i+1}\longrightarrow \cdots
\end{equation}
be a complex of locally free $\mathcal{O}_M$ sheaves. We will define the cohomology jump loci as analytic subspaces, and give their characterization. 

Since $d^i: F^i\to F^{i+1}$ is a homomorphism of locally free $\mathcal{O}_M$ modules, we can define the determinantal ideal sheaf of $m\times m$ minors to be $I_m(d_i)$. The convention is $I_m(d^i)=0$ when $m> min\{rank(F_i), rank(F_{i+1})\}$, and $I_m(d^i)=\mathcal{O}_M$ when $m\leq 0$. 

One can interpret the analytic subspace of $M$ defined by $I_m(d^i)$ as the locus where $rank(d^i)\leq m-1$. Thus, we define the $i$-th \textbf{cohomology jump ideals} as follows
$$J^i_k(F^{\bullet})\stackrel{\textrm{def}}{=}\bigcap_{a+b=l_i-k+2}\left(I_{a}(d^{i-1})+I_{b}(d^i)\right)$$
where $l_i$ is the rank of $F^i$ as locally free $\mathcal{O}_M$ module. The $i$-th \textbf{cohomology jump locus} $\Sigma^i_k(F^{\bullet})$ of the complex $F^{\bullet}$ is defined to be the closed analytic subspace of $M$ associated to the ideal sheaf $J^i_k(F^{\bullet})$. It follows immediately from the definition that $J^i_k(F^\bullet)\subset J^i_{k+1}(F^\bullet)$, or equivalently, $\Sigma^i_{k+1}(F^\bullet)\subset \Sigma^i_{k}(F^\bullet)$ as analytic subspaces of $M$. And it is easy to check that the closed points in $\Sigma^i_k(F^{\bullet})$ corresponds to the closed points $P$ on $M$ satisfying $dim H^i(F^\bullet\otimes_{\mathcal{O}_M} \mathbb{C}_P)\geq k$. The following proposition will characterize the non-reduced part of $\Sigma^i_k(F^\bullet)$.
\begin{prop}\label{csl}
Let $P$ be a point in $\Sigma^i_k(F^\bullet)\setminus \Sigma^i_{k+1}(F^\bullet)$, and let $A$ be an Artinian local ring over $\mathbb{C}$. Then for a map $s: Spec(A)\to M$, whose set theoretic image is $P$, the following two statements are equivalent
\begin{itemize}
\item[\textup{(i)}] the schematic image of $s$ lies in $\Sigma^i_k(F^\bullet)$,
\item[\textup{(ii)}] $H^i(s^{*}(F^\bullet))$ is a free module over $Spec(A)$ of rank $k$.
\end{itemize}
\end{prop}
\begin{proof}
Since the problem is local at $P$, we assume $M$ to be a closed analytic subspace of a complex ball. Denote the ring of holomorphic functions on $M$ to be $R$. By possibly shrinking the complex ball, we can assume the sheaves $F^{i-1}$, $F^i$ and $F^{i+1}$ to be the sheafification of free $R$-modules. By abusing notations, we also write $F^i$ for their global sections, which are now free $R$-modules. 

Since $P\in\Sigma^i_{k}(F^\bullet)$, but $P\notin \Sigma^i_{k+1}(F^\bullet)$, there is a unique pair of integers $(a, b)$ such that $a+b=l_i-k+2$ and $I_a(d^{i-1})+I_b(d^i)\subset m_P$, where $m_P$ is the maximal ideal of $R$ corresponding to the point $P$. In fact, suppose $(a', b')$ is another pair, and without loss of generality, suppose $a'>a, b'<b$. Then $I_{a}(d^{i-1})+I_{b'}(d^i)\subset m_P$ implies $P\in \Sigma^i_{k+b-b'}$, a contradiction to $P\notin \Sigma^i_{k+1}(F^\bullet)$.

Denote the kernel of $s^*: R\to A$ by $J_A$. Since $I_{a'}(d^{i-1})+I_{b'}(d^i)\nsubseteq m_P$ for any other pair $(a', b')\neq (a, b)$ satisfying $a'+b'=l_i-k+2$, statement (i) is equivalent to
\begin{enumerate}
\item[(i')]$I_a(d^{i-1})+I_b(d^i)\subset J_A$.
\end{enumerate}
Denote the complexes $(F^\bullet, d^\bullet)\otimes_R A$ and $(F^\bullet, d^\bullet)\otimes_R \mathbb{C}$ by $(F^\bullet_A, d^\bullet_A)$ and $(F^\bullet_\mathbb{C}, d^\bullet_{\mathbb{C}})$ respectively. Then statement (i') and statement (ii) are equivalent to 
\begin{enumerate}
\item[(i'')]$I_a(d^{i-1}_A)+I_b(d^i_A)=0$
\end{enumerate}
and
\begin{enumerate}
\item[(ii')]$H^i(F^\bullet_A)$ is a free $A$-module of rank $k$
\end{enumerate}
respectively. And the previous arguments imply that $I_{a-1}(d^{i-1})\nsubseteq m_P$ and $I_{b-1}(d^i)\nsubseteq m_P$, hence $I_{a-1}(d^{i-1}_\mathbb{C})=\mathbb{C}$ and $I_{b-1}(d^i_\mathbb{C})=\mathbb{C}$. 

We need some elementary facts about $A$-modules.

\begin{lemma}
Let $\sigma: A^\alpha\to A^\beta$ be a homomorphism of free $A$-modules, where $A$ is an Artinian local ring. Suppose $I_c(\sigma)=0$ and $I_{c-1}(\sigma\otimes_A \mathbb{C})=\mathbb{C}$. Then the image $Im(\sigma)$ is a free $A$-module of rank $c-1$.
\end{lemma}
\begin{proof}
$I_{c-1}(\sigma\otimes_A \mathbb{C})=\mathbb{C}$ implies that $\dim_\mathbb{C}Im(\sigma\otimes_A\mathbb{C})\geq c-1$. Therefore, by Nakayama's lemma, $Im(\sigma)$ contains a free $A$-submodule of rank $c-1$, which we denote by $N$. Naturally, $N$ is a submodule of $A^\beta$, the target of $\sigma$. Denote $\sigma': A^\alpha\to A^\beta/N$ to be the composition of $\sigma$ and the quotient map $A^\beta\to A^\beta/N$. Since $A^\beta/N\cong A^{\beta-c+1}$, the determinantal ideal $I_c(\sigma')$ still makes sense. Since $Im(\sigma)$ contains $N$, by choosing proper bases for $A^\alpha$ and $A^\beta$, we have obviously $I_{c}(\sigma)=I_{1}(\sigma')$. Therefore, $I_{1}(\sigma')=0$, that is $\sigma'=0$. Hence, $Im(\sigma)=N$ is a free $A$-module of rank $c-1$.
\end{proof}

Now, assuming (i''), we have $Im(d^{i-1}_A)\cong A^{a-1}$ and $Im(d^i_A)\cong A^{b-1}$. From the short exact sequence $0\to Ker(d^i_A)\to F^i_A\to Im(d^i_A)\to 0$, one deduces $Ker(d^i_A)\cong A^{l_i-b+1}$. Now, the short exact sequence $0\to Im(d^{i-1}_A)\to Ker(d^i_A)\to H^i(F^\bullet_A)\to 0$ implies $H^i(F^\bullet_A)\cong A^{l_i-a-b+2}$, which is nothing but statement (ii'). So we have showed $\text{(i'')}\Rightarrow\text{(ii')}$.

Conversely, we assume statement (ii'), i.e., $H^i(F^\bullet_A)\cong A^{l_i-a-b+2}$. By the short exact sequence $0\to Im(d^{i-1}_A)\to Ker(d^i_A)\to H^i(F^\bullet_A)\to 0$, one has $Ker(d^i_A)\cong Im(d^{i-1}_A)\oplus A^{l_i-a-b+2}$. Hence the short exact sequence $0\to Ker(d^i_A)\to F^i_A\to Im(d^i_A)\to 0$ gives rise to
\begin{equation}\label{ses}
0\to Im(d^{i-1}_A)\to A^{a+b-2}\to Im(d^{i}_A)\to 0.
\end{equation}
From the fact that $I_{a-1}(d^{i-1}_{\mathbb{C}})=\mathbb{C}$ and $I_{b-1}(d^i_\mathbb{C})=\mathbb{C}$, one immediately has $\dim_\mathbb{C}(Im(d^{i-1}_\mathbb{C}))\geq a-1$ and $\dim_\mathbb{C}(Im(d^i_\mathbb{C}))\geq b-1$. By Nakayama's lemma, $Im(d^{i-1}_A)$ contains a free $A$-submodule of rank $a-1$ and $Im(d^i_A)$ contains a free $A$-submodule of rank $b-1$. According to the short exact sequence (\ref{ses}), it immediately follows that $d^{i-1}_A\cong A^{a-1}$ and $d^i_A\cong A^{b-1}$. Therefore, statement (i'') holds. 

We have showed $\text{(i'')}\Leftrightarrow \text{(ii')}$, and hence the proposition.
\end{proof}

\begin{remark}
In \cite{dp}, the cohomology jump loci of $(F^\bullet, d^\bullet)$ is defined by the ideal $\bar{J}^i_k(F^\bullet)=I_{l_i-k+1}(d^{i-1}\oplus d^i)$. We don't know whether the two definitions are equivalent, but they define the same set of points and one can prove the same proposition for $\bar{J}^i_k(F^\bullet)$.
\end{remark}

Let $X$ be a compact K\"ahler manifold, and let $\mathcal{M}$ be the moduli space of stable vector bundles of rank $r$ with vanishing chern classes. More details about $\mathcal{M}$ will be given in next section. As the moduli space of stable vector bundles, $\mathcal{M}$ is an analytic space covered by open sets $\{U_\lambda\}$. And over each $X\times U_\lambda$, there exists a vector bundle $\mathcal{E}_\lambda$ as the universal family of vector bundles. In particular, for any point $P\in U_\lambda\subset\mathcal{M}$, as a point in the moduli space, it corresponds to the vector bundle $\mathcal{E}_\lambda|_{X\times P}$. Furthermore, by possibly passing to a finer covering, we can assume each $U_\lambda$ to be holomorphically convex. 

Denote the projections from $X\times U_\lambda$ to the first and second factors by $p_1$ and $p_2$ respectively. Let $\mathbf{R}p_{2*}(\mathcal{E}_\lambda)$ be the push-forward of $\mathcal{E}_\lambda$ by $p_2$ in the derived category. Based on Grauert's direct image theorem \cite{l}, the argument in section 5 of \cite{m} shows that $\mathbf{R}p_{2*}(\mathcal{E}_\lambda)$ can be represented by a right bounded complex $F^\bullet$ of free sheaves over $U_\lambda$ of finite rank. Then the complex $F^\bullet$ computes the cohomology of $\mathbf{R}p_{2*}(\mathcal{E}_\lambda)$. More precisely, for any closed subanalytic space $Z$ of $U_\lambda$, inducing the following diagram,
\centerline{
\xymatrix{
X\times Z\ar[r]^{i_X} \ar[d]_{p_Z}&X \times U_\lambda\ar[d]\\
Z\ar[r]^{i_0} & U_\lambda
}
}
there is a canonical isomorphism,
$$\mathbf{R}^i p_{Z*}(i_X^*(\mathcal{E}_\lambda))\cong H^i(i_o^*(F^\bullet)).$$
In particular, the closed points on $\Sigma^i_k(X)$ corresponds to the locus where the $i$-th cohomology group of the vector bundle has dimension at least $k$. 

We define the \textbf{cohomology jump loci} $\Sigma^i_{k, \,\lambda}$ on $U_\lambda$ to be analytic subspaces $\Sigma^i_k(F^\bullet)$. It can be showed that $\Sigma^i_{m, \,\lambda}$ doesn't depend on the choice of the representative $F^\bullet$ (see Remark (3.2) on page 179 of \cite{acgh}) and $\Sigma^i_{k, \,\lambda}$ patch together to form a closed analytic subspace $\Sigma^i_k(X)$ of $\mathcal{M}$. 

By applying Proposition \ref{csl} to $\mathbf{R}p_{2*}(\mathcal{E}_\lambda)$, we obtain the following corollary. 
\begin{corollary}\label{ea}
Let $P$ be a point in $\Sigma^i_k(X)\setminus \Sigma^i_{k+1}(X)$, and let $A$ be an Artinian local  ring over $\mathbb{C}$. Then for a map $s: Spec(A)\to \mathcal{M}$ whose set theoretic image is $P$, the schematic image of $s$ lies in $\Sigma^i_k(X)$ if and only if $H^i(X, E_A)$ is a free sheaf over $Spec(A)$ of rank $k$. Here $E_A$ is the pull back of $\mathcal{E}_\lambda$ by $id \times s: X\times Spec(A)\to X\times U_\lambda$, where $U_\lambda$ and $\mathcal{E}_\lambda$ are defined as above and $P\in U_\lambda$.
\end{corollary}

\section{Deformation of holomorphic vector bundles}
The moduli space of stable vector bundles over a compact K\"ahler manifold is constructed in \cite{lo} as a Hausdorff complex analytic space. In \cite{n}, Nadel studied the singularity of the components of the moduli space corresponding to vector bundles with vanishing chern classes. He showed that those components have quadratic singularities. In this section, we will review the arguments of \cite{n}, and study the deformation theory via Artinian local rings. 

Let $X$ be a compact K\"ahler manifold. We denote by $\mathcal{M}$ the moduli space of stable holomorphic vector bundles of rank $r$ with vanishing chern classes, which was constructed in \cite{lo}.  Let $E$ be a closed point in $\mathcal{M}$. By abusing notation, we also write $E$ for the corresponding stable vector bundle on $X$. The Zariski tangent space of $\mathcal{M}$ at $E$ is $T_E \mathcal{M}=H^1(X, End(E))$. Since $E$ is stable and of vanishing chern classes, the Hermitian-Einstein metric on $E$ is harmonic in the sense of \cite{s1}, i.e., the curvature of the chern connection is zero. This harmonic metric gives rise to a harmonic metric on the vector bundle $End(E)$. Therefore, according to Hodge theory, $H^1(X, End(E))$ is spanned by harmonic $(0, 1)$-forms on $End(E)$. Let's denote the harmonic forms in $H^\bullet(X, End(E))$ by $\mathcal{H}^{0,\bullet}(End(E))$. 
\begin{lemma}[Nadel, \cite{n}]\label{nadel}
The wedge product of two harmonic forms in $\mathcal{H}^{0,1}(End(E))$ is again harmonic. Therefore, the cup product 
\begin{equation}
H^1(X, End(E))\times H^1(X, End(E)) \to H^2(X, End(E))\label{cap}
\end{equation}
is isomorphic to the wedge product
$$\mathcal{H}^{0,1}(End(E))\times \mathcal{H}^{0,1}(End(E))\to \mathcal{H}^{0,2}(End(E))$$
if we identify the cohomology class and the harmonic form which represents it. 
\end{lemma}
Rewriting the map~(\ref{cap}) as a quadratic map 
$$H^1(X, End(E))\to H^2(X, End(E))$$
we denote its kernel by $C$, which is a quadratic cone in $H^1(X, End(E))$.
\begin{thm}[Nadel, \cite{n}]
The analytic germ of $\mathcal{M}$ at $E$ is isomorphic to the analytic germ of $C$ at the origin. Furthermore, there exists locally an exponential map from $C$ to $\mathcal{M}$, mapping the origin to $E$, inducing the isomorphism of the germs. The exponential map sends $[\phi]$ to the holomorphic vector bundle whose underlying $C^\infty$ bundle is $E$, but with holomorphic structure $\bar\partial_E+\phi$, where $\phi$ is a harmonic form in $\mathcal{H}^{0,1}(End(E))$, $[\phi]$ is its cohomology class and $\bar\partial_E$ is the holomorphic structure on $E$.
\end{thm}
Nadel showed that this exponential map gives locally the Kuranishi family $\mathcal{E}$ of holomorphic vector bundles. As a $C^{\infty}$ vector bundle, it is isomorphic to $p_1^*(E)$, the pull-back of $E$ by the projection $p_1: X\times C\to X$. Suppose $\bar\partial_0$ is the holomorphic structure on $\mathcal{E}$, induced by $\mathcal{E}\cong p_1^*(E)$, and suppose $\tilde\zeta$ is the tautological section of $C^{0,1}(End(\mathcal{E}))$ by identifying $C$ as a cone in $\mathcal{H}^{0,1}(End(E))$. Then according to the result of Nadel, the holomorphic structure of $\mathcal{E}$ as the Kuranishi family is $\bar{\partial}_0+\tilde{\zeta}$. 

Suppose $dim_{\mathbb{C}}H^i(X, E)=k$, the purpose of this section is to understand the local behavior of $\Sigma^i_k(X)$ near $E$ as an analytic subspace. Given an Artinian local ring $A$ and a map $s: Spec(A)\to C$, whose set theoretic image is the origin, one can pull back the Kuranishi family of vector bundles over $C$ to $Spec(A)$. Then we obtain a vector bundle $E_A$ over $X_A$, where $X_A=X\times_\mathbb{C} Spec(A)$. First of all, we give an explicit description of the holomorphic structure on $E_A$ for such a map $s:Spec(A)\to C$. The next lemma follows immediately from the property of pull-back. 
\begin{lemma}\label{cinfty}
$$C^{\infty}(E_A)= C^{\infty}(\mathcal{E})\otimes_{\mathcal{O}_C} A$$
where the tensor is over $s^*: \mathcal{O}_C\to A$. Furthermore, the holomorphic structure $\bar{\partial}_A$ on $E_A$ is induced from the holomorphic structure $\bar{\partial}_{\mathcal{E}}$, via the above isomorphism.
\end{lemma}
As $C^{\infty}$-vector bundles over $X$, $C^{\infty}(\mathcal{E})\otimes_{\mathcal{O}_C} A\cong C^{\infty}(E)\otimes_{\mathbb{C}} A$ decomposes into direct sum of copies of $C^{\infty}(E)$. More precisely, we choose a basis $\{\alpha_\mu\}_{1\leq \mu\leq l}$ for $m$, the maximal ideal of $A$, as a vector space over $\mathbb{C}$, and we let $\alpha_0=1$ in $A$. Then, $C^{\infty}(\mathcal{E})\otimes_{\mathcal{O}_C} A$ is decomposed as
\begin{equation}
C^{\infty}(\mathcal{E})\otimes_{\mathcal{O}_C} A\cong\oplus_{0\leq \mu\leq l}C^\infty(E)\,\pmb{\alpha}_\mu\label{iso}
\end{equation}
where we use symbols $\pmb{\alpha}_\mu$ to distinguish different summands of $C^\infty(E)$. Moreover, the $A$-module structure on $\oplus_{0\leq \mu\leq l}C^\infty(E)\,\pmb{\alpha}_\mu$ inherited from (\ref{iso}) is characterized by $\pmb{\alpha}_\mu=\alpha_\mu\,\pmb{\alpha}_0$.

Given $f\in C^{\infty}(E)$, it extends to a section $\tilde{f}$ of $C^{\infty}(\mathcal{E})$ by parallel extension. For any $0\leq \mu\leq l$, $f\pmb{\alpha}_\mu$ corresponds to the section $\tilde{f}\otimes \alpha_\mu$ in $C^{\infty}(\mathcal{E})\otimes_{\mathcal{O}_C} A$ via the isomorphism (\ref{iso}). Since multiplication by $\alpha_\mu$ preserves the holomorphic structure $\bar{\partial}_A$, we have
\begin{equation}
\bar{\partial}_A(f\pmb{\alpha}_\mu)=\bar{\partial}_A(\tilde{f}\otimes \alpha_\mu)=\bar{\partial}(\tilde{f})\otimes \alpha_\mu\label{equal}
\end{equation}
again via isomorphism (\ref{iso}).

Recall that $C$ is the cone in $H^1(X, End(E))$, defined as the kernel of the quadratic map $H^1(X, End(E))\to H^2(X, End(E))$. Denote the dimension of $H^1(X, End(E))$ by $q$. Let $\{\psi_\nu\}, 1\leq \nu \leq q$ be a basis of $\mathcal{H}^{0,1}(End(E))$. Then their cohomology classes $[\psi_\nu]_{1\leq \nu\leq q}$ form a basis of $H^1(X, End(E))$. Denote the coordinate functions on $H^1(X, End(E))$ with respect to the basis $\{[\psi_\nu]\}$ by $x_\nu$. Now, we can express the map $s^*: \Gamma(\mathcal{O}_C)\to A$ by $s^*(x_\nu)=\sum_{1\leq \mu'\leq l} a^{\mu'}_\nu\alpha_{\mu'}$, where $a^{\mu'}_\nu\in \mathbb{C}$.

According to the discussion about the holomorphic structure on the Kuranishi family in the paragraph preceding Lemma \ref{cinfty}, one has the following equalities,
\begin{equation*}
  \begin{split}
   \bar{\partial}(\tilde{f})\otimes \alpha_\mu&=\sum_\nu x_\nu\; \widetilde{\psi_\nu(f)}\otimes \alpha_\mu\\
   &=\sum_{\mu', \nu}\widetilde{\psi_\nu(f)}\otimes (a^{\mu'}_\nu \alpha_{\mu'})\alpha_\mu\\
   &=\sum_{\mu', \nu}a_\nu^{\mu'}\psi_\nu(f)\alpha_{\mu'}\alpha_\mu\pmb{\alpha}_0\\
   &=\sum_{\mu', \nu}a_\nu^{\mu'}\psi_\nu(f)\alpha_{\mu'}\pmb{\alpha}_\mu
  \end{split}
\end{equation*}
where $\widetilde{\psi_\nu(f)}$ is the parallel extension of $\psi_\nu(f)$ on $E$ to $\mathcal{E}$, or in other words, the pull back of $\psi_\nu(f)$ by the projection $X\times C\to X$, and the second last equality is via isomorphism (\ref{iso}). Therefore, combining this with equation (\ref{equal}), we have the following formula for the holomorphic structure $\bar{\partial}_A$.

\begin{prop}\label{cs}The holomorphic structure $\bar\partial_A$ on $E_A$ is given by
\begin{equation*}
\bar{\partial}_A(f\pmb{\alpha}_\mu)=\bar\partial(f)\pmb\alpha_\mu+\sum_{\substack{1\leq\mu'\leq l\\ 1\leq \nu\leq q}}a_\nu^{\mu'}\psi_\nu(f)\;\alpha_{\mu'}\pmb{\alpha}_\mu.
\end{equation*}
\end{prop}


Under the same notations, we define the subspace of $H^1(X, End(E))$ spanned by $[\sum_{\nu}a_\nu^{\mu'}\psi_\nu]_{1\leq \mu' \leq l}$ to be the \textbf{vector space of derivatives} of $E_A$, denoted by $D(E_A)$. Notice that the dual vector space $H^1(X, End(E))^{\vee}$ is contained in the ring $\Gamma(\mathcal{O}_C)$. Therefore, $s^*: \Gamma(\mathcal{O}_C)\to A$ restricts to a linear map $g: H^1(X, End(E))^{\vee}\to m$, since $s^*$ sends the maximal ideal at origin to $m$. Thus we get $g^{\vee}: m^{\vee}\to H^1(X, End(E))$. By definition, one can easily check that $D(E_A)=Im(g^{\vee})$, which doesn't depend on the choices of $x_{\nu}$ or $\alpha_{\mu'}$.

The cup product map, composing with $End(E)\times E\to E$, defines naturally a bilinear map
$$H^1(X, End(E))\times H^\bullet(X, E)\to H^{\bullet+1}(X, E)$$
which, from now on, will be simply called the ``cup product" map.

\begin{prop}\label{main}
Suppose $dim_{\mathbb{C}}H^i(X, E)=k$. Then the following statements are equivalent, 
\begin{enumerate}
\item[\textup{(i)}] $H^i(X, E_A)$ is a free $A$-module of rank $k$,

\item[\textup{(ii)}] under the cup product maps $H^1(X, End(E))\times H^i(X, E)\to H^{i+1}(X, E)$ and $H^1(X, End(E))\times H^{i-1}(X, E)\to H^i(X, E)$, 
$D(E_A)$ is contained in the annihilators of $H^i(X, E)$ and $H^{i-1}(X, E)$ respectively. In other words, the images of $D(E_A)\times H^i(X, E)$ and $D(E_A)\times H^{i-1}(X, E)$ are both zero.

\item[\textup{(iii)}] let $\xi_\lambda, 1\leq\lambda\leq k$ be a basis of $E$-valued harmonic $(0, i)$-forms. Then $E_A$-valued $(0, i)$-forms $\xi_\lambda\pmb\alpha_\mu$ are $\bar\partial_A$-closed and $[\xi_\lambda\pmb\alpha_\mu]_{1\leq \lambda\leq k, \,0\leq \mu\leq l}$ forms a basis of $H^i(X, E_A)$. Similarly, let $\zeta$ be any $E$-valued harmonic $(0, i-1)$-form. Then $\zeta\pmb\alpha_\mu$ is $\bar\partial_A$-closed and the cohomology classes $[\zeta \pmb\alpha_\mu]_{0\leq \mu\leq l, \,\zeta\in \mathcal{H}^{0, i-1}(E)}$ span $H^{i-1}(X, E_A)$. 
\end{enumerate}
\end{prop}
Before proving this Proposition, we need a few lemmas.

Since $E$ is stable and of vanishing chern classes, there exists a Hermitian-Einstein metric $h$ on $E$, whose curvature is zero. This metric $h$ induces Hermitian-Einstein metrics on $End(E)=E\otimes E^{\vee}$ and $End(E)\otimes E$ respectively. Fixing these metrics on $End(E)$ and $End(E)\otimes E$ respectively, we have the following lemma, which is very similar to Lemma 3.1 of \cite{gl2} and the statement (*) after Theorem 1.1 of \cite{n}.

\begin{lemma}\label{hf}
Let $\theta$ be an $End(E)$-valued harmonic $(0, 1)$-form and let $\xi$ be an $E$-valued harmonic $(0, i)$-form. Under the natural contraction map $End(E)\otimes E\to E$, we can compose them to a $E$-valued $(0, i+1)$-form, which we write as $\theta(\xi)$. Then $\theta(\xi)$ is also harmonic.
\end{lemma}
\begin{proof}
The contraction map $\tau: End(E)\otimes E\to E$ is canonically isomorphic to the tensor product of $E^{\vee}\otimes E\to \mathcal(O)_X$ and $E$. Since $\mathcal{O}_X$ is a direct summand of $E^\vee\otimes E$ and the direct sum respects the Hermitian-Einstein metrics, $\tau: End(E)\otimes E\to E$ is projection to a direct summand in the category of Hermitian-Einstein vector bundles. Therefore, $\tau: End(E)\otimes E\to E$ is compatible with taking conjugation with respect to the Hermitian-Einstein metrics. In other words, for any $End(E)\otimes E$-valued form $\zeta$, $\widebar{\tau(\zeta)}=\tau(\bar{\zeta})$, where the bars are taking conjugation with respect to the fixed Hermitian-Einstein metrics on $End(E)\otimes E$ and $E$ respectively.

On a Hermitian-Einstein vector bundle, whose curvature is zero, the notion of $\partial$ harmonic and $\bar{\partial}$ harmonic are equivalent due to Hodge theory. Moreover, since star operator commutes with taking conjugate, a bundle valued form $\rho$ is $\partial$ harmonic $\Leftrightarrow \rho$ and $*\rho$ are $\partial$-closed $\Leftrightarrow$ $\bar{\rho}$ and $*\bar{\rho}$ are $\bar\partial$-closed $\Leftrightarrow$ $\bar\rho$ is $\bar\partial$ harmonic. Therefore, an $E$-valued (or $End(E)$-valued) form  is harmonic if and only if its conjugation is harmonic. 

Now $\theta$ is a harmonic $End(E)$-valued $(0, 1)$-form and $\xi$ is a harmonic $E$-valued $(0, i)$-form $\Rightarrow$ $\bar\theta$ is a holomorphic $End(E)$-valued 1-form and $\bar\xi$ is a  holomorphic $E$-valued $i$-form $\Rightarrow$ $\bar\theta(\bar\xi)=\widebar{\theta(\xi)}$ is a holomorphic $E$-valued $(i+1)$-form $\Rightarrow$ $\theta(\xi)$ is a harmonic $E$-valued $(0, i+1)$-form.
\end{proof}

We can always filter the maximal ideal $m$ of $A$ by ideals $m=m_0\supset m_1\supset \cdots \supset m_{l-1}\supset m_{l}=0$, such that $m_j/m_{j+1}\cong \mathbb{C}$. Given the filtration, we can choose the basis $\{\alpha_j\}_{1\leq j\leq l}$ of $m$ compatible with the filtration, that is, $m_j$ is spanned by $\alpha_{j+1}, \alpha_{j+2}, \ldots, \alpha_{l}$ as a $\mathbb{C}$ vector space. Denote the quotient ring $A/m_j$ by $A_j$ for $0\leq j\leq l$. Considering $E_A$ as a sheaf of $A$-modules on $X$, we denote $E_A\otimes_A A_j$ by $E_j$. For example, $E=E_0$ and $E_l=E_A$. By taking tensor product of $E_A$ and $0\to \mathbb{C}\to A_{j+1}\to A_{j}\to 0$ over $A$, we obtain a short exact sequence $0\to E\to E_{j+1}\to E_{j}\to 0$ for each $0\leq j\leq l-1$. 

As $C^\infty$ vector bundles, $E_{j}$ is a direct sum of $j+1$ copies of $E$. Let $E_{j}= E\bar{\pmb\alpha}_0\oplus E\bar{\pmb{\alpha}}_1\oplus\cdots \oplus E\bar{\pmb\alpha}_{j}$ be the $C^\infty$ decomposition which is compatible with the decomposition $E_A=E\pmb\alpha_0\oplus\cdots\oplus E\pmb\alpha_l$. More precisely, we require that under the projection map $E_A\to E_{j}$, symbol $\pmb\alpha_h$ maps to $\bar{\pmb\alpha}_h$ if $0\leq h\leq j$, and zero otherwise. Let the corresponding $C^\infty$ decomposition of vector bundle $\mathscr{H}om(E_{j}, E)$ be $\mathscr{H}om(E_{j}, E)= End(E)\pmb\beta_0\oplus\cdots \oplus End(E)\pmb\beta_{j}$. By Proposition \ref{cs}, under the same notation, the holomorphic structure $\bar\partial$ on $\mathscr Hom(E_j, E)$ is given by

\begin{equation}\label{cs2}
\bar\partial(g\pmb\beta_h)=\bar\partial(g)\pmb\beta_h+\sum_{\substack{1\leq\mu\leq j\\1\leq\nu\leq q}}a^{\mu}_{\nu}\psi_\nu(g)\alpha_\mu\pmb\beta_h
\end{equation}

where $g$ is any $C^\infty$ section of $End(E)$ and $\psi_\nu(\cdot)$ is defined to be mutiplication by $\psi_\nu$ on right. Since $\mathscr Hom(E_j, E)$ has an $A$-module structure induced from the one on $E_j$, multiplication by $\alpha_\mu$ is well defined.

\begin{lemma}\label{ec}
The extension class of the short exact sequence $0\to E\to E_{j+1}\to E_j\to 0$ is represented by the $\mathscr Hom(E_j, E)$-valued 1-form 
$$\sum_{0\leq h\leq j}\sum_{1\leq \nu\leq q}a^{h+1}_\nu \psi_\nu \pmb\beta_h.$$
\end{lemma}
\begin{proof}
It follows from next lemma and formula (\ref{cs2}) for the holomorphic structure of $\mathscr Hom(E_j, E)$.
\end{proof}

Given a short exact sequence of holomorphic vector bundles $0\to V'\to V\to V''\to 0$ on $X$,  let $\epsilon: V''\to V$ be a $C^\infty$ map which splits $0\to V'\to V\to V''\to 0$ in the category of $C^\infty$ bundles. Let $\bar\partial$ be the holomorphic structure on $\mathscr Hom(V'', V)$ induced from the ones on $V''$ and $V$. Considering $\epsilon$ as a $C^\infty$ section of $\mathscr Hom(V'', V)$, $\bar\partial(\epsilon)$ is a $\mathscr Hom(V'', V)$-valued 1-form. Under the induced map $\mathscr Hom(V'', V)\to \mathscr Hom(V'', V'')$ by $V\to V''$, the image of $\bar\partial(\epsilon)$ is zero, because the identity section in $\mathscr Hom(V'', V'')$ is holomorphic. Therefore, $\bar\partial(\epsilon)$ lifts to a $\mathscr Hom(V'', V')$-valued 1-form, which we denote by $\bar\partial(\epsilon)_0$. 

\begin{lemma}\label{lec}
$\bar\partial(\epsilon)_0$ is $\bar\partial$-closed, and it represents the extension class of \\
$0\to V'\to V\to V''\to 0$.
\end{lemma}
\begin{proof}
Take the Dolbeault resolution of the short exact sequence 
$$0\to \mathscr Hom(V'', V')\to \mathscr Hom(V'', V)\to \mathscr Hom(V'', V'')\to 0.$$
Then the lemma follows from the standard arguments in sheaf cohomology.
\end{proof}

By taking quotient of the inclusion map $m_{j-1}\to A$ by $m_j$, we have a natural homomorphism $\mathbb{C}\to A_j$ of $A$-modules. Consider the map $E\to E_j$ induced by taking tensor product of $E_A$ and $\mathbb{C}\to A_j$ over $A$. Under the above notation, the map sends a section $f$ of $E$ to the section $f\bar{\pmb\alpha}_j$ of $E_j$. Therefore, the induced map on cohomology $H^{i-1}(X, E)\to H^{i-1}(X, E_j)$ sends $[\zeta]$ to $[\zeta\bar{\pmb\alpha}_j]$ for any harmonic form $\zeta\in \mathcal{H}^{0, i-1}(E)$. Composing with the boundary map $\delta:H^{i-1}(X, E_j)\to H^i(X, E)$, the image of $[\zeta]$ would be $[\sum_{1\leq \nu\leq q}a_\nu^{j+1}\psi_\nu(\zeta)]$ by Lemma \ref{ec}. To summarize, we have the following lemma,

\begin{lemma}\label{bm}
Under the composition $H^{i-1}(X, E)\to H^{i-1}(X, E_j)\to H^i(X, E)$, where the first map is induced by the map of vector bundles $E\to E_j$ and the second is the boundary map $\delta: H^{i-1}(X, E_j)\to H^i(X, E)$, the image of $[\zeta]$ is $[\sum_{1\leq \nu\leq q}a_\nu^{j+1}\psi_\nu(\zeta)]$. 
\end{lemma}

\begin{proof}[Proof of Proposition \ref{main}]
We use induction on the dimension of $A$ as a complex vector space. In particular, we assume that the proposition is true for $E_{l-1}$. 

First, we show that $\text{(iii)}\Rightarrow \text{(i)}$. Since $C^\infty(E_A)$ has an $A$-module structure and multiplying by $\alpha_\mu\in A$ preserves the holomorphic structure, we have $\alpha_\mu[\xi_\lambda\pmb\alpha_0]=[\xi_\lambda \alpha_\mu\pmb\alpha_0]=[\xi_\lambda\pmb\alpha_\mu]$ in the $A$-module $H^i(X, E_A)$. Therefore, $H^i(X, E_A)$ is generated by $[\xi_\lambda\pmb\alpha_0]$, and according to (iii), there is no relation between the generators. Therefore, $H^i(X, E_A)$ is free of rank $k$.

Next, we show that $\text{(i)}\Rightarrow \text{(ii)}$. Assuming (i), we know that the dimension of $H^i(X, E_A)$ is $(l+1)k$. From the long exact sequence associated to $0\to E\to E_{j+1}\to E_j\to 0$, we have $\dim H^i(X, E_{j+1})\leq \dim H^i(X, E_j)+\dim H^i(X, E)$ for all $0\leq j\leq l-1$. Therefore, $\dim H^i(X, E_A)=(l+1)k$ implies that the equality holds for every $0\leq j \leq l-1$. Thus, in particular,  the map $H^i(X, E_A)\to H^i(X, E_{l-1})$ is surjective. Since $H^i(X, E_A)$ is a free $A$-module of rank $k$, $H^i(X, E_{l-1})$ as an $A_{l-1}$-module is generated by $k$ elements. On the other hand, $\dim H^i(X, E_A)=\dim H^i(X, E_{l-1})+\dim H^i(X, E)$, and hence $\dim H^i(X, E_{l-1})=lk$. Therefore, $H^i(X, E_{l-1})$ must be a free $A_{l-1}$-module. 

By induction hypothesis, since statement (i) holds for $E_{l-1}$, (iii) must hold too. Hence $[\xi_\lambda\bar{\pmb\alpha}_\mu]$, $1\leq \lambda\leq k, 0\leq \mu\leq l-1$ forms a basis of $H^i(X, E_{l-1})$. As a consequence of $\dim H^i(X, E_{l})=\dim H^i(X, E_{l-1})+\dim H^i(X, E)$, the boundary map $\delta:H^i(X, E_{l-1})\to H^{i+1}(X, E)$ associated to the short exact sequence $0\to E\to E_A\to E_{l-1}\to 0$ is zero. According to Lemma \ref{ec}, the boundary map $\delta$ is the cup product with 
$$\sum_{0\leq h\leq l-1}\sum_{1\leq \nu\leq q}a^{h+1}_\nu \psi_\nu \pmb\beta_h.$$
Therefore 
$$0=\delta([\xi_\lambda \bar{\pmb\alpha}_{\mu-1}])=[\sum_{1\leq \nu\leq q}a^{\mu}_\nu \psi_\nu(\xi_\lambda)]$$
for every $1\leq \mu\leq l$. Thus, we have showed that under the cup product map 
$$H^1(X, End(E))\times H^i(X, E)\to H^{i+1}(X, E)$$
the image of $([\sum_{1\leq \nu\leq q}a_\nu^{\mu}\psi_\nu], [\xi_\lambda])$ is zero for any $1\leq \mu\leq l$ and $1\leq \lambda\leq k$. Hence the image of $D(E_A)\times H^i(X, E)$ is zero, that is the first half of statement (ii).

For the other half, we consider the composition 
$$H^{i-1}(X, E)\to H^{i-1}(X, E_j)\to H^i(X, E)$$
 as in Lemma \ref{bm}. By counting dimensions as above, the boundary map $\delta: H^{i-1}(X, E)\to H^i(X, E_{j})$ has to be zero. Therefore, the image of $[\zeta]$ under the composition has to be zero, that is, $[\sum_{1\leq \nu\leq q}a^{j+1}_\nu\psi_\nu(\zeta)]=0$ for all $0\leq j\leq l-1$ and $\zeta\in \mathcal{H}^{0, i-1}(E)$. Since $[\sum_{1\leq \nu\leq q}a^{j+1}_\nu\psi_\nu]_{0\leq j\leq l-1}$ span $D(E_A)$ and the cohomology classes of the harmonic forms in $\mathcal{H}^{0, i-1}(E)$ span $H^{i-1}(X, E)$, the image of $D(E_A)\times H^{i-1}(X, E)$ under the cup product map is zero. 

Finally, we show that $\text{(ii)}\Rightarrow \text{(iii)}$. According to first half of statement (ii), the image of $D(E_A)\times H^i(X, E)$ is zero under the cup product map. Hence the cohomology class $[\sum_{1\leq \nu\leq q} a^\mu_\nu \psi_\nu(\xi_\lambda)]=0$, for any $1\leq \mu\ \leq l, 1\leq\lambda\leq k$. By Lemma \ref{hf}, $\psi_\nu(\xi_\lambda)$ are harmonic, hence so are there linear combinations. Thus, $\sum_{1\leq \nu\leq q} a^\mu_\nu \psi_\nu(\xi_\lambda)=0$ for every $1\leq \mu\leq l$. According to Lemma \ref{cs}, 
$$\bar\partial_A(\xi_\lambda \pmb\alpha_\mu)=\bar\partial(\xi_\lambda)+\sum_{\substack{1\leq \mu'\leq l\\ 1\leq \nu\leq q}}a^{\mu'}_\nu\psi_\nu(\xi_\lambda)\alpha_{\mu'}\pmb\alpha_\mu.$$
On the right side of the equation, $\bar\partial(\xi_\lambda)=0$, because $\xi_\lambda$ is harmonic. And the second term is also zero, because we have just showed $\sum_{1\leq \nu\leq q} a^\mu_\nu \psi_\nu(\xi_\lambda)=0$ for every $1\leq \mu\leq l$. Therefore, $\xi_\lambda \pmb\alpha_\mu$ is $\bar\partial_A$-closed for every $1\leq \mu\leq l$. Since the image of $D(X_A)\times H^{i-1}(X, E)$ under the cup product map is zero, the same argument shows that $\zeta\pmb\alpha_\mu$ is $\bar\partial_A$-closed for any harmonic form $\zeta\in \mathcal{H}^{0, i-1}(E)$. 

We have assumed (ii) is true for $E_A$. Obviously, $D(E_{l-1})\subset D(E_A)$, and hence (ii) is true for $E_{l-1}$. Hence the induction hypothesis implies that $[\xi_\lambda \bar{\pmb\alpha}_\mu]_{1\leq \lambda\leq k, 0\leq \mu\leq l-1}$ form a basis of $H^i(X, E_{l-1})$, and $[\zeta\bar{\pmb\alpha}_\mu]_{0\leq \mu\leq l-1, \zeta\in \mathcal{H}^{0, i-1}(E)}$ span $H^{i-1}(X, E_{l-1})$. We have showed in the previous paragraph that 
$\sum_{1\leq \nu\leq q}a_\nu^\mu\psi_\nu(\xi_\lambda)=0$ for every $1\leq \mu\leq l$, and with the same argument we can show $\sum_{1\leq \nu\leq q }a_\nu^\mu\psi_\nu(\zeta)=0$ for any $\zeta\in \mathcal{H}^{0, i-1}(E)$. Therefore, by Lemma \ref{ec}, the boundary maps $\delta^{i-1}: H^{i-1}(X, E_{l-1})\to H^i(X, E)$ and $\delta^i: H^i(X, E_{l-1})\to H^{i+1}(X, E)$ are both zero maps. Thus, the map $H^{i-1}(X, E_{A})\to H^{i-1}(X, E_{l-1})$ is surjective, and part of the long exact sequence is now a short exact sequence $0\to H^i(X, E)\to H^i(X, E_A)\to H^i(X, E_{l-1})\to 0$. 

Since $[\xi_\lambda\bar{\pmb\alpha}_\mu]\in H^i(W, E_{l-1})$ is the image of $[\xi_\lambda\pmb\alpha_\mu]\in H^i(X, E_A)$ under the map $H^i(X, E_A)\to H^i(X, E_{l-1})$ and the image of $[\xi_\lambda]\in H^i(X, E)$ under $H^i(X, E)\to H^i(X, E_A)$ is $[\xi_\lambda\pmb\alpha_l]$, and since $0\to H^i(X, E)\to H^i(X, E_A)\to H^i(X, E_{l-1})\to 0$ is exact, $[\xi_\lambda\pmb\alpha_\mu]_{1\leq \lambda\leq k, 0\leq \mu\leq l}$ form a basis of $H^i(X, E_A)$. Similarly, since the image of $[\zeta\pmb\alpha_\mu]$ under the map $H^{i-1}(X, E_A)\to H^{i-1}(X, E_{l-1})$ is $[\zeta\bar{\pmb\alpha}_\mu]$ and the image of $[\zeta]$ under the map $H^{i-1}(X, E)\to H^{i-1}(X, E_A)$ is $[\zeta\pmb\alpha_l]$, and since $H^{i-1}(X, E)\to H^{i-1}(X, E_A)\to H^{i-1}(X, E_{l-1})\to 0$ is exact, $[\zeta\pmb\alpha_\mu]_{0\leq \mu\leq l, \zeta\in \mathcal{H}^{0, i-1}(E)}$ span $H^{i-1}(X, E_A)$. We finished the proof of $\text{(ii)}\Rightarrow \text{(iii)}$.
\end{proof}

Now, we are ready to prove Theorem \ref{linear}. Let $H\subset H^1(X, End(E))$ be the intersection of the two annihilators defined in Proposition \ref{main} (ii). 

\begin{proof}[Proof of Theorem \ref{linear}]
Let $A$ be an Artinian local ring. According to Corollary \ref{ea}, for any map $s:Spec(A)\to U$ whose set-theoretic image is at $P$, the schematic image of $Spec(A)$ lies in $\Sigma^i_k(X)$ if and only if $H^i(X, E_A)$ is a free $A$-module of rank $k$. And by $\text{(i)}\Leftrightarrow \text{(ii)}$ in Proposition \ref{main}, $H^i(X, E_A)$ is a free $A$-module of rank $k$ if and only if the space of derivatives $D(E_A)$ lies in $H$. Let $s': Spec(A)\to U'$ be the composition of $s: Spec(A)\to U$ and the isomorphism between $U$ and $U'$. Then the schematic image of $s$ lies in $\Sigma^i_k(X)$ if and only if the schematic image of $s'$ lies in $H\cap U'$. 

Therefore, $\Sigma^i_k(X)$ at $P$ and $H\cap U'$ at origin represent the same functor from the category of Artinian local rings to the category of sets. Denote the origin in $U'$ by $O$. Thus, the completion ring $\widehat{(\mathcal{O}_{\Sigma^i_k})_P}$ is isomorphic to $\widehat{(\mathcal{O}_{H\cap U'})_O}$, and the isomorphism is induced by the isomorphism between $\mathcal{O}_U$ and $\mathcal{O}_{U'}$. Therefore, the isomorphism between $U$ and $U'$ induces an isomorphism between the analytic germs $(\Sigma^i_k)_P$ and $(H\cap U')_O$, hence by possibly shrinking $U$ and $U'$, $\Sigma^i_k(X)\cap U$ is isomorphic to $H\cap U'$ under the isomorphism between $U$ and $U'$ induced by the exponential map.

\end{proof}

\end{document}